\documentclass{amsart}
\usepackage{amssymb}
\usepackage{amsmath,amsfonts,amsthm,amstext}
\usepackage[all]{xy}
\newtheorem{theorem}{Theorem}[section]
\newtheorem{lemma}[theorem]{Lemma}

\newtheorem{cor}[theorem]{Corollary}

\numberwithin{equation}{section}

\begin{document}

\title{Self-similar groups of type $FP_n$}

\author{Dessislava H. Kochloukova, Said N. Sidki }

\address{State University of Campinas (UNICAMP), Campinas, Brazil; Federal University of Bras\'ilia (unB), Bras\'ilia, Brazil} 
\email{}

\thanks{The first names author was supported by  "bolsa de produtividade em pesquisa" CNPq 303350/2013-0, Brazil and by grant "projeto regular" FAPESP 2016/05678-3, Brazil. The second named author was supported by  grant FAPESP 2016/05271-0, Brazil for a month visit to UNICAMP in September, 2016.}


\date{}


\begin{abstract} We construct new classes of self-similar groups :  $S$-aritmetic groups, affine groups and metabelian groups. Most of the soluble ones are finitely presented and of type $FP_n$ for appropriate $n$. 
\end{abstract}

\maketitle

\section{Introduction} In this paper we construct new classes of self-similar groups : $S$-aritmetic groups, affine groups and metabelian groups. 
To fix the terminology,  a group $G$ is self-similar if  it is a state closed  subgroup of the automorphism group of an infinite regular one-rooted  $m$-tree ${\mathcal T}_m$ (i.e. every vertex has precisely $m$ descendents) that preserves the levels of the tree. And if the action of $G$ on the first level of ${\mathcal T}_m$ is transitive we say that $G$ is a transitive self-similar group . Obviously if $m = 2$ every nontrivial self-similar group is a transitive self-similar group but if $m \geq 3$ there are self-similar groups that are not transitive self-similar. A group acting on the tree ${\mathcal T}_m$ is finite-state provided each of its elements has a finite number of states. An  automata group is a finitely generated self-similar finite-state group. We say that $G$ is a transitive automata group if it is transitive as a self-similar group. We use virtual
endomorphisms introduced by Nekrashevych and Sidki in \cite{Nekra}, \cite{N-S} to establish the
classes of groups we consider are transitive self-similar

There is an extensive literature on self-similar groups and the subclass of automata groups (i.e. finitely generated, finite state, self-similar groups) including the Grigorchuk group \cite{Gri}, the Gupta-Sidki group \cite{Gu-S}, abelian groups \cite{BruSid}, finitely generated nilpotent groups \cite{Ber-Sidki}, soluble groups \cite{BarthSunik} and
 aritmetic groups \cite{Kapo}. For a good survey on this topic the reader is referred to the book  \cite{Nekr}.

 Most of the self-similar groups considered in the literature are not finitely presented, while in this paper most of the new classes of transitive self-similar groups we construct  are finitely presented and of homological type $FP_n$ for appropriate $n$. Examples of  metabelian self-similar groups with higher finiteness homological properties  were first  introduced by Bartholdi, Neuhauser and Woess in \cite{BNW}. These  examples  now belong to the class
 of S-aritmetic nilpotent-by-abelian groups defined in
   Theorem A. 

Recall from homological group theory that a group $G$ has a homological type $FP_n$ if the trivial $\mathbb{Z} G$-module $\mathbb{Z}$ has a projective resolution with all projectives finitely generated in dimension $\leq n$. A group is of type $FP_1$ if and only if it is finitely generated. Every finitely presented group is of type $FP_2$ but as shown in \cite{BB} the converse does not hold. 
It is easy to see that the property $FP_n$ is invariant under going up and down by finite index, this and more properties  of groups of type $FP_n$ can be found in  \cite{Bieri}, \cite{Brown}.  We rely on already existing
results in the literature on higher finiteness homological properties to check the groups in question are of type $FP_n$  and finitely presented.

In Section \ref{S-arithmetic} we construct a class of $S$-aritmetic  self-similar groups. Given a commutative algebra $A$ we denote  the  upper triangular $m \times m$ matrix group, or Borel subgroup, with coefficients in $A$ by $U(m, A)$ and its projective quotient by $PU(m, A)$ which is nilpotent-by-abelian.

\medskip
{\bf Theorem A} {\it Let $n \geq 1$ be an integer, $p$ be a prime number and $$A = \mathbb{F}_p[x^{\pm 1}, {1 \over {f_1}}, \ldots , {1 \over {f_{n-1}}}]$$
the subring of ${ \mathbb F}_p(x)$, where $f_0 = x, f_1, \ldots , f_{n-1} \in \mathbb{F}_p[x] \setminus \mathbb{F}_p (x-1)$ are pairwise different, monic, irreducible polynomials.  Then $PU(m,A)$ is 
transitive, finite-state and state-closed of degree $p^{l}$, where $l$ is a cubic polynomial of $m$. 
}

\medskip
The group $G = PU(m,A)$ considered in Theorem A is of type $F_{n}$ (i.e. is finitely presented if $n \geq 2$ and of homological type $FP_{n}$) but not of type $FP_{n+1}$. This follows from the fact that  $U(m, A)$ has these properties as  shown by Bux in \cite{Bux} using the theory of buildings and  by the fact that the center of $U(m,A)$ is a finitely generated abelian group, so is of type $FP_{\infty}$ ( i.e. is $FP_s$ for every $s$).

  In Section \ref{sectionN-S} we introduce a positive characteristic version of the affine linear groups $\mathbb{Z}^n \rtimes B(n, \mathbb{Z})$ considered by Nekrashevych and Sidki in \cite{N-S}, where
$B(n, \mathbb{Z})$ is the (pre-Borel) subgroup of finite index in  $GL_n(\mathbb{Z})$ consisting of  the matrices whose entries above the main diagonal are even. They showed that these affine groups are realizable as transitive automata groups acting on the binary tree. We define here similarly
$B(n, \mathbb{F}_p[x])$ as the subgroup of $GL_n(\mathbb{F}_p[x])$ consisting of  the matrices whose entries above the main diagonal belong to the ideal  $(x-1) \mathbb{F}_p[x]$.

\medskip
{\bf Theorem B} {\it Let $n \geq 3$ be an integer, $p$ a prime number and $G$ the affine group  $ \mathbb{F}_p[x]^n \rtimes B(n, \mathbb{F}_p[x])$.
Then $G$ is transitive, finite-state and state-closed of degree $p$.}

\medskip
 For $n = 2$ the group $G$ in Theorem B is state closed of degree $p$ but not finitely generated since 
by a result of Nagao \cite{Nag} $SL_2(\mathbb{F}_p[x])$ is not finitely generated.  The fact that $GL_n( F_p[x])$ is finitely generated for $n \geq 3$ follows from work of Behr in \cite{Behr} and  this implies that $G$ is finitely generated for $n \geq 3$.  
As for the higher finiteness homological properties, Abramenko proved in \cite{Abr} that $SL_n(\mathbb{F}_q[x])$ is $FP_{n-2}$ if ${\mathbb F}_q$ is the finite field of order $q = p^s$ for sufficiently large $s$. The restriction on $s$ was removed by Bux, K\"ohl, Witzel in \cite{Bux-K-W}. 

Note that in both Theorem A and Theorem B the polynomial $x-1$ can be substituted with an irreducible, monic polynomial $g \in \mathbb{F}_p[x]$. In Theorem A we need the additional condition that $g$ be coprime to the polynomials $f_0, \ldots, f_n$. In this case the degree of the representation is $p^l$, where $l$ is a cubic polynomial in $m$ and whose coefficients depend on $deg(g)$. Furthermore, the degree of the representation of $G$ in Theorem B becomes $p^{deg(g)}$. 

The rest of the paper is devoted  to the class of metabelian groups. It is worth noting that not every finitely generated metabelian group is transitive self-similar. In \cite{Alex-Said2} Dantas and Sidki showed that  $\mathbb{Z} \wr \mathbb{Z}$ is not a transitive self-similar group.
In Section \ref{Construction} we  make  explicit calculations about the states of some generating elements of a metabelian self-similar group that has finite index in $PU(2, A)$ from Theorem A. 

In Section \ref{Krull-dim} we take up the study of self-similar representations of metabelian groups $B \rtimes Q$, where $B$ and $Q$ are abelian and  $B$ has 
 Krull dimension 1 as $\mathbb{Z} Q$-module via conjugation. 

\medskip
{\bf Theorem C} {\it 
Let $Q$ be a finitely generated abelian group and  $B$ be a finitely generated, right $\mathbb{Z} Q$-module of Krull dimension 1  such that $C_Q(B) = \{ q \in Q \mid B(q-1) = 0 \} = 1$. Then $G = B \rtimes Q$ is a transitive self-similar group. }

\medskip
 The proof of Theorem C uses methods from commutative algebra. The general strategy of the proof 
goes as follows: given $\delta$ in $\mathbb{Z} Q$, we may form the subgroup $H=(\delta B) \rtimes Q$.  We show the existence of $\delta$ in $\mathbb{Z} Q$ such that $\delta B$ is of finite index in $B$ and such that the map $f(\delta b) = b$ for all $b$ in $B$ together with  $f|_Q=id_Q$ defines a simple virtual endomorphism $f : H \to G$ of $G$. With this, we have the extra nice property that $f$ is an isomorphism between the subgroup $H$ of finite index in $G$ and $G$ itself.
Note that by the main result of Dantas and Sidki \cite{Alex-Said2}, Theorem C does not hold for metabelian groups, where the Krull dimension of $B$ is bigger than 1.

Under certain additional conditions the group $G$ from Theorem C is finitely presented and of type $FP_m$.  In Section \ref{section-meta} we recall the Bieri-Strebel  theory of $m$-tame modules and the  $FP_m$-Conjecture for metabelian groups.
If in Theorem C, the subgroup $B$ is $m$-tame as $\mathbb{Z}Q$-module and is either  ${\mathbb Z}$-torsion free or of finite exponent, then $G$ is a transitive self-similar metabelian group of type $FP_m$, since in both cases the $FP_m$-Conjecture for metabelian groups holds by \cite{Aberg} and \cite{Desi}. 

Both Theorem A and Theorem C generalize the automata group considered by Bartholdi, Neuhauser and Woess in
 \cite[p.~16]{BNW}. The Cayley graph of this group is a Diestel-Leader graph i.e.  is the horocyclic product of several homogeneous trees. We note that the idea of horocyclic product  of trees was used earlier in \AA{}berg's paper on metabelian groups of 1986  with different terminology  \cite{Aberg}.

In Section \ref{local}  using localization technique from commutative algebra we embed  the transitive self-similar group $G = C_p \wr \mathbb{Z}^d$, constructed by  Dantas and Sidki in  \cite{Alex-Said}, in a transitive self-similar group $\widetilde{G}$, whose definition depends on a choice of a polynomial $g$.   This new class  of metabelian transitive self-similar groups  $\widetilde{G}$ is of the form $\widetilde{A} \rtimes \widetilde{Q}$ and the Krull dimension of $\widetilde{A}$ is $d$, so it can be arbitrary large, contrary to the class of groups considered in Theorem C, where the Krull dimension is 1.

\medskip
{\bf Theorem D} {\it Let $p$ be a prime number, $g \in \mathbb{F}_p[x] \setminus ( \cup_{j \geq 0} \mathbb{F}_p x^j  \cup (x-1) \mathbb{F}_p[x])$ and $\widetilde{G}$ be the group $  \widetilde{A} \rtimes \widetilde{Q}$, where  $$\widetilde{A} = \mathbb{F}_p[x_1^{\pm 1}, {1 \over g(x_1)}] \otimes_{\mathbb{F}_p} \ldots \otimes_{\mathbb{F}_p} \mathbb{F}_p[x_d^{\pm 1}, {1 \over g(x_d)}]$$   and $$\widetilde{Q} = \mathbb{Z}^{2d} = \langle x_1, \ldots, x_d, y_1, \ldots, y_d \rangle$$ and $y_i$ and $x_i$ act on $\widetilde{A}$ (via conjugation) as multiplication by $g(x_i)$ and $x_i$ respectively. Then $\widetilde{G}$ is a transitive self-similar group. }

\medskip
Note that for special choices of the polynomial $g$, for example irreducible and coprime to $x$, the group $\widetilde{G}$ from Theorem D is finitely presented; this follows from the Bieri-Strebel classification of finitely presented metabelian
groups in \cite{BS}, see Theorem \ref{cond-met}.  But since $\widetilde{A}$ is never 3-tame as $\mathbb{Z} \widetilde{Q}$-module, the group $\widetilde{G}$ is never of type $FP_3$.

\section{Preliminaries}

  \subsection{Preliminaries on transitive self-similar groups and virtual endomorphisms} 

Given a homomorphism $f : H \to G$ from a subgroup $H$ of finite index in $G$, we call $f$ a virtual endomorphism. 
  Now we explain how the existence of a simple virtual endomorphism $$f : H \to G$$ produces a faithful, state-closed action of $G$ on the infinite regular one-rooted  $m$-tree ${\mathcal T}_m$, where $m = [G : H]$ and this action is transitive on the first level of the tree.  The endomorphism is simple if every normal subgroup $K$ of $G$ contained in $H$ and such that $f(K) \subseteq K$ is trivial. This idea was first used in \cite{N-S} and later on was applied several times to construct transitive self-similar groups in \cite{Ber-Sidki}, \cite{Alex-Said}, \cite{Alex-Said2}.
  
  Let $T = \{ t_0, \ldots, t_{m-1} \} $ be a right transversal  of $H$ in $G$ i.e. $G$ is the disjoint union $\cup_{t \in T} H t$. Note that $G$ acts  on the right coset  classes of $H$ in $G$ by right multiplication.  Given $g \in G$ we can define recursively the action of $g$ on ${\mathcal T}_m$ as
\begin{equation} \label{form}
g = (g_0, \ldots, g_{m-1}) \sigma \in G \wr S_m,
\end{equation}
where $\sigma$ is a permutation of the set $\{ 0, 1, \ldots, m-1 \}$ defined by the permutation action of $g$ on the coset classes i.e.  $H t_i g = H t_{(i)\sigma}$  and 
 we call $$t_0 g t_{(0)\sigma}^{-1}, \ldots,  t_{m-1} g t_{(m-1)\sigma}^{-1}$$ the cofactors of $g$.
The states $g_0, \ldots, g_{m-1}$ are defined  by $$g_i = f( t_i g t_{(i)\sigma}^{-1}) \in G, \hbox{ for every }i \in X = \{ 0, \ldots, m-1 \}.$$ 
This is a recursive definition and we can repeat this process with each state $g_0, \ldots, g_{m-1}$.
The product of elements written as in (\ref{form}) is given by the following formula
\begin{equation} \label{product}
(g_0, \ldots, g_{m-1}) \sigma (\tilde{g}_0, \ldots, \tilde{g}_{m-1}) \tilde{\sigma} = (g_0 \tilde{g}_{(0) \sigma}, g_1 \tilde{g}_{(1) \sigma}, \ldots, g_{m-1} \tilde{g}_{(m-1)\sigma})  \sigma \tilde{\sigma},
\end{equation}
where $\sigma, \tilde{\sigma} \in Perm(X)$. The degree of this presentation of $G$ as a self-similar group is the index $[G : H]$.

The converse is also true. If $G$ is a transitive self-similar group i.e.  it has a faithful, state-closed action  on the tree ${\mathcal T}_m$  and the action of $G$ on the vertices that are descendants of the root i.e. the set $X = \{0, 1, \ldots, m-1 \}$ is transitive, then we can define a subgroup $H$ of finite index in $G$  as the the stabilizer of $0 \in X$ in $G$. As before there is a recursive description  of the elements $g$ of $G$
\begin{equation} \label{recursive23} 
g = (g_0, g_1, \ldots, g_{m-1}) \sigma,
\end{equation}
where $\sigma$ is a permutation from $S_m$ and $g_0, \ldots, g_{m-1}$ are elements of $G$, with each $g_i$ acting on the subtree of ${\mathcal T}_m$  headed by the vertex $i \in X$. Thus $g \in H$ precisely when $(0)\sigma = 0$. Then the virtual endomorphism  $f : H \to G$ can be defined by  $f(g) = g_0$. 

Finally for $g$ as in (\ref{recursive23}) we define recursively the set of states  ${\mathcal Q}(g)$ of the elements $g$ as  the  subset of $Aut({\mathcal T}_m)$ given by
$$
{\mathcal Q}(g) =  \{ g \} \cup \cup_{0 \leq i \leq m-1} {\mathcal Q}(g_i).
$$
 A self-similar group  $G$ i.e. a state closed subgroup of $Aut({\mathcal T}_m)$ is an automata group   (or a finite state self-similar group) if there is a finite generating set $X$ of $G$ such that for every $g \in X$ we have that ${\mathcal Q}(g)$ is finite. If furthermore in the definition of automata group $G$ is transitive self-similar, we say that $G$ is a transitive automata group.  
\subsection{Preliminaries on metabelian groups of type $FP_m$} \label{section-meta}

Several of our examples are metabelian groups i.e. there is a short exact sequence $A \to G \to Q$ such that  $G$ is finitely generated and $A$ and $Q$ are abelian. Then we can view $A$ as a right $\mathbb{Z} Q$-module via conjugation and $A$ is finitely generated as $\mathbb{Z} Q$-module since $Q$ is finitely presented. The set $Hom(Q, \mathbb{R}) \setminus \{ 0 \}$ has an equivalence relation $\sim$ given by multiplication with a positive real number i.e. for $\chi_1, \chi_2 \in Hom (Q, \mathbb{R}) \setminus \{ 0 \} $ we have $\chi_1 \sim \chi_2$ if and only if there is a positive real number $r$ such that $\chi_1 = r \chi_2$. 
Define
$$
S(Q) =   Hom(Q, \mathbb{R}) \setminus \{ 0 \} / \sim.
$$ 
Note that for the torsion free rank $n$ of $Q$ we have that $Hom (Q, \mathbb{R}) \simeq \mathbb{R}^n$, hence 
$$
S(Q)  \simeq {S^{n-1}},
$$
where $S^{n-1}$ is the unit sphere in ${\mathbb{R}}^n$. We write $[\chi]$ for the equivalence class of $\chi \in Hom(Q, \mathbb{R}) \setminus \{ 0 \}$ and define the monoid
$$
Q_{\chi} = \{ q \in Q \mid \chi(q) \geq 0 \}.
$$
The Bieri-Strebel invariant defined in \cite{BS} is
$$
\Sigma_A(Q) = \{ [\chi] \mid A \hbox{ is finitely generated as } \mathbb{Z} Q_{\chi}\hbox{-module} \}
$$
and was used in the classification of finitely presented metabelian groups in \cite{BS}.

\begin{theorem} \cite{BS} \label{cond-met}
Let $A \to G \to Q$ be a short exact sequence of groups with $G$ finitely generated, $A$ and $Q$ abelian. Then the following are equivalent :

1) $G$ is finitely presented;

2) $G$ is of homological type $FP_2$;

3) $\Sigma_A(Q) \cup - \Sigma_A(Q) = S(Q)$.
\end{theorem}  

We denote the complement $S(Q) \setminus \Sigma_A(Q)$ as $\Sigma_A^c(Q)$ and say that $A$ is $m$-tame as $\mathbb{Z} Q$-module if for every $[\chi_1], \ldots, [\chi_m] \in \Sigma_A^c(Q)$, not necessary different, we have $\chi_1 + \ldots + \chi_m \not= 0$. 
The last condition of Theorem \ref{cond-met} is equivalent to $A$ being 2-tame as $\mathbb{Z} Q$-module. The following conjecture suggests when a metabelian group is of type $FP_m$.

\medskip
{\bf The $FP_m$-Conjecture for metabelian groups}
{\it Let $A \to G \to Q$ be a short exact sequence of groups with $G$ finitely generated, $A$ and $Q$ abelian. Then $G$ is of type $FP_m$ if and only if $A$ is $m$-tame as $\mathbb{Z} Q$-module.}

\medskip
We state below cases when the $FP_m$-Conjecture  holds.
Recall that a group $G$ is of finite Pr\"ufer rank if there is an integer $d$ such that every finitely generated subgroup can be generated by at most $d$ elements. If $G$ is an extension of $A$ by $Q$ with both $A$ and $Q$ abelian and $G$ finitely generated it is easy to see that $G$ is of finite Pr\"ufer rank precisely when $dim_{\mathbb{Q}} (A \otimes_{\mathbb{Z}} \mathbb{Q}) < \infty$ and the torsion part $tor(A)$ of $A$ is finite.    

 \begin{theorem} \label{fpm-met} The $FP_m$-Conjecture holds when

\medskip
1. $G$ is of finite Pr\"ufer rank;

or 

2. $A$ is of finite exponent and of Krull dimension 1.
\end{theorem}

Part 1  of Theorem \ref{fpm-met}  was proved in \cite{Aberg} and part 2 was proved in \cite{Desi}.

\subsection{Horocyclic product of trees and \AA{}berg's construction} \label{trees}

Here we discuss the link between the construction of horocyclic product of trees and \AA{}berg's construction. The first example  of a self-similar metabelian group of type $FP_n$ was given by Bartholdi, Neuhauser, Woess in \cite{BNW}. The group is
$$
\mathbb{F}_q[x, {1 \over x}, {1 \over {x+1}}, \ldots , {1 \over {x+ q-1}}] \rtimes Q
$$
where $Q = \langle x_0, \ldots , x_{q-1} \rangle \simeq \mathbb{Z}^q$, $x_i$ acts by conjugation via multiplication with $f_i = x+ i$, $q = p_1 \ldots p_r$ is the decomposition as product of powers of different primes, so each $p_i$ is a power of a prime and $\mathbb{F}_q = {\mathbb F}_{p_1} \times \ldots \times {\mathbb F}_{p_r}$, where ${\mathbb F}_{p_i}$ is the field of $p_i$ elements.  
By \cite{BNW} the Cayley graph of $G$ is the horocyclic product of $d = q+1$ homogeneous trees i.e. each tree comes with a height function, the horocyclic product is the set of $d$-tuples in the direct product of the trees whose coordenate heights sum up to zero.

In this subsection we  explain the connection between the idea of horocyclic product of trees for metabelian groups and \AA{}berg's paper \cite{Aberg}, where the $FP_m$-Conjecture for metabelian groups was proved for groups $G$ of finite Pr\"ufer rank.  The main idea of \cite{Aberg} is to consider finitely many trees on which the metabelian group $G$ acts. Each tree is associated to a valuation and the trees are Serre's valuation trees. Each tree is equipped with a height  function and inside the direct product of these trees  is considered the subcomplex $Y$ of points whose height functions sum to  0. It was shown in \cite{Aberg} that  the action of $G$ on $Y$ has the following properties :

  1. the action is combinatorial in the sense that if $g \in G$ fixes a cell of $Y$ then it fixes the cell pointwise;
  
  2. the cell stabilizers are all polycyclic;
  
  3. $G$ acts cocompactly on $Y$ in each dimension i.e. there are finitely many $G$-orbits  of cells in each dimension;
  
  4. if the Bieri-Strebel invariant $\Sigma_A(Q)$ is $m$-tame, then $Y$ is $(m-1)$-connected.
  
   Then by the Brown criterion \cite{Brown1} $G$ is of type $FP_m$.  Indeed, the conditions in the Brown criterion require that the action is combinatorial,  stabilizers of cells of dimension $i$ are of type $FP_{m-i}$ for $0 \leq i \leq m-1$, $G$ acts cocompactly on the $m$-skeleton of $Y$ and the complex $Y$  is $m-1$ connected. Note that the horocyclic product of the trees in the \AA{}berg construction is exactly the 1-skeleton of the complex $Y$.

\subsection{Preliminaries from commutative algebra} \label{comm}

Recall some basic facts from commutative algebra, for more details see \cite[Chapter~3]{eisenbud} or \cite{algebrabook}. 
Let $S$ be a commutative Noetherian ring with unit.
Then $S$ has finitely many minimal prime ideals and its intersection is the nilpotent radical $\sqrt{0} = \{ s \in S \mid s^n = 0 \hbox{ for some } n \geq 1 \}$. Hence for every ideal $I$ of $S$ we have that
$$\sqrt{I} = \{ s \in S \mid s^n \in I \hbox{ for some } n \geq 1 \} = \widetilde{P}_1 \cap \ldots \cap \widetilde{P}_j,
$$ 
where $\widetilde{P}_1, \ldots, \widetilde{P}_j$ are the minimal prime ideals of $S$ above $I$.

 Let $V$ be a finitely generated, right $S$-module. An 
{{associated prime}} of $V$ is a prime ideal $P$ of $S$ such that 
 $P=ann_S(m)$ for some element  $m\in V$.  The set of associated primes of $V$ is always finite and is denote by $Ass(V)$. The minimal primes among the prime ideals containing $ann_S(V)$ are associated primes of $V$.
For the elements $s \in S$ such that for some $v \in V \setminus \{ 0 \}$, $v s = 0$, we have
$$
s \in \cup_{P \in Ass(V)} P.
$$
There is a composition series
$$
V_0 = 0 \subseteq V_1 \subseteq \ldots \subseteq V_{s-1} \subseteq V_s = V,
$$
such that  each  factor $V_{i+1} / V_i \simeq  S/ P_i$ for some prime ideal $P_i$ of $S$ and furthermore
$$
Ass(V) \subseteq \{ P_0, \ldots, P_{s-1} \} \subseteq Supp (V),
$$
where $Supp(V)$ is the set of all prime ideals  $P$ of $S$ such that  the localization module $V(S \setminus P)^{-1} \not= 0$.
Furthermore the minimal elements of  the three sets of primes considered above $
Ass(V)$, $ \{ P_0, \ldots, P_{s-1} \} $ and $ Supp (V)
$, are the same.

  \section{Nilpotent-by-abelian $S$-aritmetic self-similar groups : proof of Theorem A} \label{S-arithmetic}
  
  Recall that a group $G$ is of homotopical type $F_s$ if there is a $K(G,1)$-complex with finite $s$-skeleton. Thus type $F_1$ is finite generation and $F_2$ is finite presentability. For $s \geq 2$ type $F_s$ is equivalent to type $FP_s$ together with finite presentability.  
A matrix $B = (b_{i,j}) \in M_n(R)$ is upper triangular (or upper Borel), where $R$ is a ring, if $b_{i,j} = 0$ for $i > j$. 
  
\begin{theorem} \label{ThmBux} \cite{Bux} Let ${\mathcal G}$ be a Chevalley group, ${\mathcal B}$ its Borel subgroup, $K$ a finite field extension of $k(x)$, where $k$ is a finite field and $S$ a finite set of places. Let ${\mathcal O}_S = \{ t \in K \mid v(t) \geq 0 \hbox{ for } v \notin S \}$. Then ${\mathcal B}({\mathcal O}_S)$ has type $F_{d-1}$ but not $FP_d$, where $d = | S |$.
\end{theorem} 

In this section we consider applications of Theorem \ref{ThmBux} 
 for ${\mathcal G} = GL_m$. Then ${\mathcal B} = UT$, where $T$ is the group of diagonal  matrices in ${\mathcal G}$  and $U$ are the upper triangular matrices with  1 on the main diagonal. Fix a prime $p$, $k = \mathbb{F}_p$ the finite field with $p$ elements and $K = k(x)$. Let $$f_0 = x, f_1, \ldots, f_{n-1}$$ be pair-wise different, monic, irreducible polynomials in $\mathbb{F}_p[x] \setminus \mathbb{F}_p (x-1)$. For an irreducible polynomial $f \in \mathbb{F}_p[x] \setminus  \mathbb{F}_p$ define $v_f$ the $f$-adic valuation i.e. $v_f(g) = s$ for $g \in \mathbb{F}_p[x]$  if $s \geq 0$ and $f^s$ divides $g$ but $f^{s+1}$ does not divide $g$. Let $v_0$ be the  valorization defined by $v_0(g) = - deg(g)$. Set $$S = \{ v_0, v_{f_i} \mid 0 \leq i \leq n-1 \}.$$ Then $$A : = {\mathcal O}_S = \mathbb{F}_p[x^{\pm 1}, {1 \over f_1}, \ldots, {1 \over f_{n-1}}]$$ and  
$$G_0 = {\mathcal B}({\mathcal O}_S) = N \rtimes Q,$$
where $N = U(A)$, $Q = ( \mathbb{F}_p^* \times Q_0) \times \ldots \times (\mathbb{F}_p^* \times Q_0) = ( \mathbb{F}_p^* \times Q_0 )^m$, $\mathbb{F}_p^* =\mathbb{F}_p \setminus \{ 0 \}$  and $Q_0$ is the subgroup of $A \setminus \{ 0 \}$ generated by $f_0, \ldots, f_{n-1}$. Then 
$Z(G_0)$ is the group of diagonal matrices $\lambda I_m$, where $\lambda \in \mathbb{F}_p^* \times Q_0$. Set
$$
G = G_0/ Z(G_0).
$$

By Theorem \ref{ThmBux} since $|S| = n+1$ we obtain that $G_0$ is of type $F_n$ but not of type $FP_{n+1}$. Since $Z(G_0) \simeq \mathbb{F}_p^* \times Q_0 $ is a finitely generated abelian group, so is of type $FP_{\infty}$ and is finitely presented, we deduce the following result.

\begin{cor} For  $n \geq 2$, $G$ is $F_n$ but not $FP_{n+1}$. In particular $G$ is finitely presented.  \end{cor}

Define $I = (x - 1)A$ an ideal in $A$.

\begin{lemma} $I \not= A$ and $A = I \oplus {\mathbb F}_p$ as vector space.
\end{lemma}
\begin{proof}
Consider the ring $B = {\mathbb{F}_p}[x]/(x-1) \simeq \mathbb{F}_p $.   The images $g_0, g_1, \ldots, g_{n-1}$ of $f_0 = x, f_1, \ldots , f_{n-1}$ in $B$  are all non-zero. Hence
$$
A/ I \simeq B[{1 \over {g_0}}, \ldots, {1 \over {g_{n-1}}} ]  = B \not= 0.
$$ 
Note that $\mid B \mid = p < \infty$.
\end{proof}
  
    Define $N_0$ the subgroup of the nilpotent group $N$ that contains the matrices $(a_{i,j})$ such that $a_{i,j} \in  I^{j-i}$. Define the group homomorphism
$$
\theta : N_0 \to N
$$
given  by $$\theta(M)= \widetilde{M}, \hbox{ where }M = (a_{i,j}),  \widetilde{M} = ( \widetilde{a}_{i,j}) \hbox{ and } \widetilde{a}_{i,j} = { a_{i,j} \over (x-1)^{j-i}}.$$ 
The rule of multiplication of matrices implies that $\theta$ is a homomorphism.
Set $H = N_0 \rtimes Q$ a subgroup of $G_0$ and 
 define 
$$ 
f : H \to G_0,
$$
where the restriction of $f$ on $Q$ is the identity and the restriction of $f$ on $N_0$ is $\theta$.
\begin{lemma} $f$ is a virtual endomorphism of groups and $Z(G_0)$ is the maximal normal and $f$-invariant subgroup $K$ of $G_0$ which  is contained in  $H$.
Hence $G = G_0/ Z(G_0)$ is a transitive self-similar group of degree $p^l$, where $l$ is a cubic polynomial of $m$.
\end{lemma}

\begin{proof}
Note first that $[G_0 : H] = p^{m-1} p^{2(m-2)} \ldots p^{i(m-i)} \ldots p^{m-1} = p^l < \infty$, where $l = \sum_{1 \leq i \leq m} i(m-i)$ is a cubic polynomial of $m$. The rule of multiplication of matrices implies that $f$ is a homomorphism.

Suppose that $K$ is a normal subgroup of $G_0$ contained in $H$  such that $f(K) \subseteq K$. Set $K_0= K \cap N$. Then $f(K_0) \subseteq K_0$ and for $A = (a_{i,j}) \in K_0$ we have that $A \in U$ and   $a_{i,j} \in \cap_{s \geq 1} I^s = 0$ for $i < j$. Thus $A$ is the identity matrix $I_m$ and $K_0$ is the trivial group.
 
Since $N$ intersects $K$ trivially and both are normal in $G$ we deduce that $N K \simeq N \times K$, so $K \subseteq C_{G_0}(N) = \{ g \in G_0 \mid [g, N] = 1 \} = \langle Z(G_0), Z(N) \rangle = Z(G_0) \times Z(N)$. Since for  $(\alpha, \beta) \in Z(G_0) \times Z(N)$ we have $f(\alpha, \beta) = (\alpha, \beta / (x-1)^{m-1})$ we deduce that for $(\alpha, \beta) \in K$ we have $\beta = 1_{G_0}$. Hence $K \subseteq Z(G_0)$ as required.

Finally since $Z(G_0)$ are the scalar matrices $\lambda I_m$, where $\lambda \in \mathbb{F}_p^*  \times Q_0$ and $I_m$ is the identity matrix, we have that $Z(G_0)  \subseteq Q$. Hence $f(Z(G_0)) = Z(G_0)$.
 \end{proof}

Note that we have constructed a self-similar $S$-arithmetic nilpotent-by-abelian group $G = G_0/ Z(G_0)$. The following result proves  Theorem A.

\begin{theorem} The group $G$ is a transitive, automata group.
\end{theorem}
\begin{proof}
For $1 \leq i < j \leq m$ let $x_i^{(j)}$ be the diagonal matrix $diag(1, \ldots, 1, f_j, 1, \ldots, 1)$, where $f_j$ is on the $i$th place and let  $u_{i}$ be the matrix with entries 1 on the main diagonal, 0 below the main diagonal and above the main diagonal  all entries 0 except on the $i,i+1$-th place, where the entry is 1. 
 Note that the ideal  $I^j = (x-1)^j A$ of $A$ has a transversal $T_j$ in $A$ defined  by
 $$
 T_j  = \{ g \in \mathbb{F}_p[x] \mid deg(g) \leq j-1 \hbox{ or } g = 0\}.
 $$
 We fix the following right transversal $T$ of $H$ in $G$ 
 $$
 T = \{ (a_{i,j}) \in M_m(\mathbb{F}_p[x])  \mid$$ $$ a_{i,j} = 0 \hbox{ if }1 \leq j < i \leq m, a_{i,i} = 1 \hbox{ for } 1 \leq i \leq m, a_{i,j} \in T_{j-i} \hbox{ if } 1 \leq i < j \leq m \}.
 $$  
 Thus $G_0$ is the disjoint union $\cup_{t \in T} H t$.
 Note that
 $$Tu_i = T \hbox{ for every } 1 \leq i \leq m-1.
 $$ 
 Hence the set of states $\mathcal Q(u_i) $ contains only the trivial group.
 
 \medskip
 {\bf Claim 1} Let $t \in T$. Then $t^{-1} \in T $. 
 
 \medskip
{\bf Proof of Claim 1} We use induction on $m$.
 Set
 $
 t = A = (a_{i,j}) \in T
 $
 a matrix $m \times m$.
 Thus $a_{i,i} = 1$ for all $1 \leq i \leq m$, $a_{i,j} = 0$ for $1 \leq j < i \leq m$ and $deg(a_{i,j}) \leq j -i -1$ for $1 \leq i < j \leq m$ if $a_{i,j} \not= 0$.
 
 Let $B = A^{-1} = (b_{i,j})$. Claim 1 is equivalent to $b_{i,i} = 1$ for all $1 \leq i \leq m$, $b_{i,j} = 0$ for $1 \leq j < i \leq m$ and $deg(b_{i,j}) \leq j -i -1$ for $1 \leq i < j \leq m$ if $b_{i,j} \not= 0$. The first two properties obviously hold since $A$ is an upper triangular matrix.
Set $M = (m_{i,j})$ a matrix $(m-1) \times (m-1)$, where $m_{i,j} = a_{i+1, j+1}$ for $ 1 \leq i,j \leq  m-1$. Then since $A$ is an upper triangular matrix for $\widetilde{M} = (\widetilde{m}_{i,j})$ a matrix $(m-1) \times (m-1)$ defined by $\widetilde{m}_{i,j} = b_{1+i, 1+j}$ we have that $\widetilde{M} = M^{-1}$. 
By induction on $m$ we have that 
\begin{equation} \label{deg1} deg(\widetilde{m}_{i,j}) \leq j-i-1 \hbox{ if }1 \leq i<j \leq m-1 \hbox{ and }\widetilde{m}_{i,j} \not= 0.
\end{equation}

Similarly we can define a $(m-1) \times (m-1)$  matrix $S = (s_{i,j})$ such that $s_{i,j} = a_{i,j}$ and a $(m-1) \times (m-1)$  matrix $\widetilde{S} = (\widetilde{s}_{i,j})$ such that $\widetilde{s}_{i,j} = b_{i,j}$, where $ 1 \leq i,j \leq m-1$. Then again since $A$ is an upper triangular matrix we have that $\widetilde{S} = S^{-1}$ and by induction on $m$ we have that 
\begin{equation} \label{deg2} deg(\widetilde{s}_{i,j}) \leq j-i-1 \hbox{ if } 1 \leq i<j \leq m-1 \hbox{ and }\widetilde{s}_{i,j} \not= 0.
\end{equation} 
 By (\ref{deg1}) and (\ref{deg2}) we obtain that
 \begin{equation} \label{deg3} 
 deg(b_{i,j}) \leq j-i-1 \hbox{ for } i < j, (i,j) \not= (1,m) \hbox{ and } b_{i,j} \not= 0.
 \end{equation}  
 Finally it remains to consider the coefficient $b_{1,m}$. Note that since $A B = I_m$ we have
 \begin{equation} \label{deg4}
 0 = b_{1,m} + a_{1,2} b_{2,m} + \ldots + a_{1,m-1} b_{m-1,m} + a_{1,m}. 
 \end{equation}
 If $a_{1,i} b_{i,m} \not= 0$ and $ 2 \leq i \leq m-1$ we have
 \begin{equation} \label{deg5}
 deg(a_{1,i} b_{i,m}) = deg(a_{1,i}) + deg(b_{i,m}) \leq i-2 + m-i-1 = m -3 < m-2.
 \end{equation}
 Finally  (\ref{deg4}) and (\ref{deg5}) together with $deg(a_{1,m}) \leq  m-2$ if $a_{1,m} \not= 0$ imply that
 $$
 def(b_{1,m}) \leq m-2 \hbox{ if } b_{1,m} \not= 0.
 $$
 This completes the proof of Claim 1.
 
 \medskip
 Consider the finite set
 $$
 \Delta_{k}^{(s)} = \{ (a_{i,j}) \in M_m(\mathbb{F}_p[x])  \mid a_{i,j} = 0 \hbox{ if }i > j, a_{i,i} = 1 \hbox{ for } 1 \leq i \not= k \leq m, $$ $$ a_{k,k} = f_s, deg(a_{i,j}) \leq deg(f_s) \hbox{ or } a_{i,j} = 0 \hbox{ if } i < j  \}
 $$ 
 
 \medskip
 {\bf Claim 2} The set of states ${\mathcal Q}(x_k^{(s)}) \subseteq \Delta_{k}^{(s)}$. In particular  ${\mathcal Q}(x_k^{(s)})$ is finite.
 
 \medskip
 {\bf Proof of Claim 2}
 Since $x_k^{(s)} \in \Delta_{k}^{(s)}$ to complete the proof of the claim it suffices to show that
 $$\cup_{\delta \in \Delta_{k}^{(s)}} {\mathcal Q}( \delta) \subseteq \Delta_{k}^{(s)}.$$
Let $\delta \in \Delta_{k}^{(s)}$. We will calculate $\delta_1, \ldots, \delta_d$, where $d = [G_0 : H]$, such that the recursive description of $\delta$ is
$$
\delta = (\delta_1, \ldots, \delta_d) \sigma$$
for some $\sigma \in S_d$. We aim to show that $\delta_r \in \Delta_k^{(s)}$ for $1 \leq r \leq d$. 

 Suppose $t_1 \in T$. Then there is an unique $t_2 \in T$  such that $t_1 \delta t_2^{-1} \in H$. Write
 $$
 t_1 = (a_{i,j}), \ \delta = ( b_{i,j}), \ t_2^{-1}  = ( c_{i,j}), t_1 \delta t_2^{-1} = (d_{i,j}).
 $$ 
By Claim 1  $t_2^{-1}$ is an upper triangular matrix, hence all three matrices $t_1, \delta$ and $t_2^{-1}$ are upper triangular.
 Then
 $$d_{i,j} = \sum_{i \leq r \leq u \leq j} a_{i,r} b_{r,u} c_{u,j} \hbox{  for } i \leq j
 $$
 and $d_{i,j} = 0$ for $i > j$.
 Then the states $\delta_1, \ldots, \delta_d$ of $\delta$ are
 $$
 f( t_1 \delta t_2^{-1}) =  ( e_{i,j}).
 $$
 Thus by the definition of the virtual endomorphism $f$
 we have
 $$
 e_{i,j} = d_{i,j}/ (x-1)^{j-i} \hbox{ for } i \leq j
 $$
 and $e_{i,j} = 0$ otherwise. We claim that 
 \begin{equation} \label{delta-inclusion} 
 (e_{i,j}) \in \Delta_{k}^{(s)}.
 \end{equation}
 Note that the diagonal of $(e_{i,j})$ is the same as the diagonal of $(d_{i,j})$, and is equal to $(1, \ldots, 1, f_s, 1, \ldots, 1)$, where $f_s$ is on the $k$th place. Finally we find an upper limit for $deg(e_{i,j})$ when $i < j$ and this will imply (\ref{delta-inclusion}).
 Note that if $e_{i,j} \not= 0 $
 \begin{equation} \label{deg01} deg(e_{i,j}) = deg (d_{i,j}) - (j-i) 
 \end{equation}
 and
 \begin{equation} \label{deg02}
 deg(d_{i,j}) \leq max_{i \leq r \leq u \leq j;  a_{i,r} b_{r,u}  c_{u,j}\not=0} \{ deg (a_{i,r}) + deg(b_{r,u}) + deg(c_{u,j}) 
 \}. 
 \end{equation}
 Observe that $deg(b_{r,u} ) \leq deg(f_s) $ if $b_{r,u} \not= 0$, then
 \begin{equation} \label{deg03}
 deg (a_{i,r}) + deg(b_{r,u}) + deg(c_{u,j}) \leq (r - i) + deg(f_s) + j - u \leq deg(f_s) + (j-i).
 \end{equation}
Then  by (\ref{deg01}), (\ref{deg02}) and (\ref{deg03})
$$
deg(e_{i,j}) \leq  deg(f_s) \hbox{ or } e_{i,j} = 0 \hbox{ for } i < j, \hbox{so the matrix  } (e_{i,j}) \in  \Delta_k^{(s)}.
$$
This completes the proof of Claim 2.
\end{proof}

\section{A positive characteristic version of the Nekrashevych-Sidki example : proof of Theorem B} \label{sectionN-S}

1. We recall first the Nekrashevych-Sidki example from \cite{N-S}. Let $W = \mathbb{Z}^n$. $B(n, \mathbb{Z})$ is the subgroup of $GL_n(\mathbb{Z})$ containing the matrices containing only even elements above the main diagonal. The elements of $W$ are considered as columns and we have
$$
G = W \rtimes B(n , \mathbb{Z}),
$$
where for $b_1, b_2 \in B(n , \mathbb{Z}), w_1, w_2 \in W$ we have $$(w_1 b_1) (w_2 b_2) = (w_1 + b_1 (w_2)) (b_1 b_2)$$  and $$ b_1(w_2) \hbox{ is the product of matrices }b_1 w_2.$$
Then for
$$H = (2\mathbb{Z}   \times \mathbb{Z}^{n-1}) \rtimes B(n, \mathbb{Z})$$ Nekrashevych and Sidki defined in \cite{N-S} a virtual endomorphism $$f : H \to G$$ given by 
$$
f (wb) = A (w) Ab A^{-1} \hbox{ for } w \in 2 \mathbb{Z} \times {\mathbb{Z}^{n-1}}, b \in B(n, \mathbb{Z}),
$$				
where $A$ is the matrix 
\[ A = \begin{pmatrix} 0  &  1  & 0&\cdots & 0 \\
                    0 &0&1& \cdots &0\\
																				\cdots & \cdots  & \cdots & \cdots   & \cdots  \\
                     0 & 0 &   \cdots & 0&  1 \\
                    {1 \over 2}& 0 & \cdots &0 &0 \end{pmatrix}\]
and $A(w)$ is the product of matrices $A w$. By \cite{N-S} $G$ is a finite state self-similar group i.e. an automata group.

2. Let $p$ be a prime number and ${\mathbb{F}}_p$ be the field with $p$ elements.
Define $B(n, \mathbb{F}_p[x])$ as the subgroup of matrices of $GL_n( \mathbb{F}_p[x])$ whose elements on $(i,j)$th place are inside the ideal  $I = (x-1) \mathbb{F}_p[x]$ of $\mathbb{F}_p[x]$, whenever $i < j$. 
Now define $V = (\mathbb{F}_p[x])^n$ and think of the elements of $V$ as columns. We define the group $G$ as
$$
G = V \rtimes B(n, \mathbb{F}_p[x])
$$
with multiplication given by
\begin{equation} \label{product010}
(v_1, b_1) (v_2, b_2) = (v_1 + b_1(v_2), b_1 b_2 ),
\end{equation}
where $b_1, b_2 \in B(n, \mathbb{F}_p[x])$, $v_1, v_2 \in V$ and $b_1(v_2)$ is the product of matrices $b_1 v_2$.
Define $$
V_0 = (x-1) \mathbb{F}_p[x] \times \mathbb{F}_p[x]^{n-1}
$$ and 
$$
H = V_0 \rtimes B(n, \mathbb{F}_p[x]).
$$
Thus $H$ is a subgroup of $G$ of index $p$. Define a virtual endomorphism
$$
f : H \to G
$$
given by
$$
f(v,b) = (A(v), Ab A^{-1}),
$$
where 
\[ A = \begin{pmatrix} 0  &  1  & 0&\cdots & 0 \\
                    0 &0&1& \cdots &0\\
																				\cdots & \cdots  & \cdots & \cdots   & \cdots  \\
                     0 & 0 &   \cdots & 0&  1 \\
                    {1 \over {x-1}}& 0 & \cdots &0 &0 \end{pmatrix}\]
                 and $A(v)$ is the product of matrices $Av$.
								
From now on for simplicity we write $B$ for $B(n, \mathbb{F}_p[x])$.

\begin{lemma} $f$ is a simple endomorphism.
\end{lemma}
\begin{proof}
Let $K$ be a normal subgroup of $G$ contained in $H$ such that $f(K) \subseteq K$.  Let $M = K \cap V_0$, thus $M$ is a $\mathbb{Z} B$-submodule of $V_0$  such that
$$
f(M) \subseteq M.
$$
Let $m = (m_1, \ldots, m_n)^t$ be a matrix column from $M$, where upper index $t$ denotes matrix transpose. Then by the definition of $f$ we have that
$
f(m) = A(m) =  (m_2, \ldots, m_n, m_1/ (x-1))^.
$
Then $$f^n(m) = (m_1/ (x-1), \ldots, m_n/(x-1))^t,$$
hence $m_i \in \cap_j (x-1)^j \mathbb{F}_p[x] = 0$, so $M = 0$. Then since $B$ does not contain non-trivial elements that centralize $V$ we  deduce that $K$ is the trivial subgroup of $G$ i.e. $f$ is simple.
\end{proof}

We fix $e_i = (0, \ldots, 0, 1, 0 \ldots, 0)^t$ the matrix column that has unique non-zero entry 1 on the $i$th position.

\begin{lemma} Let $n \geq 3$ and $Y_0$ be a finite generating set of $B$. Then for every $g \in \{ e_1, \ldots, e_n \} \cup Y_0$ we have that the set of states ${\mathcal Q}(g)$ is finite. In particular  $G = V \rtimes B$ is a transitive, automata group.
\end{lemma}
 
\begin{proof} Observe that $\{ e_1 \} \cup B$ generates $G$. Indeed if $v \in V$ there are $v_1, \ldots, v_n \in V$ such that the $i$th coordinate of $v_i$ is 1. Then there is $b_i \in B$ such that the first column of $b_i$ is $v_i$ and so $\sum_{1 \leq i \leq n} b_i(e_1) = \sum_{1 \leq i \leq n} v_i = v$. Thus $v$ belongs to the subgroup of $G$ generated by $e_1$ and $B$.

Recall that by \cite{Behr} $B$ is finitely generated for $n \geq 3$.
Hence $G$ is generated by  $\{ e_1\} \cup Y_0$, where $Y_0$ is a finite  generating set of $B$.

Define for any non-zero matrix $C = (c_{i,j})$ with entries in $\mathbb{F}_p[x]$ the non-negative integer
$$
\rho(C) = max \{ deg(c_{i,j}) \mid c_{i,j} \not= 0 \}.
$$    
If $C$ is the zero matrix we define $\rho(C) = - \infty$.
Using the definition of $\rho$ we define
$$
\Delta_k = \{ (v,b) \in V \rtimes B \mid 
\rho(A^j b A^{-j}) \leq k \hbox{ for } 0 \leq j \leq n-1 \hbox{ and  } \rho(v)  \leq k \}.
$$ 
Observe that $A^n = {1 \over {x-1}} I_n$, where $I_n$ is the identity $n \times n$-matrix, hence $A^n b A^{-n} = b$ and the condition on  $A^j b A^{-j}$ in the definition of $\Delta_k$ holds for all $j \in \mathbb{Z}$.

 Define a right transversal $T$ of $H$ in $G$ by
$$
T = \{ e_1^{\alpha}  \mid \alpha \in \mathbb{F}_p \}.
$$

\medskip
{\bf Claim} For every $g \in \Delta_k$ the set of states ${\mathcal Q}(g)$ is a subset of $\Delta_k$. In particular, since $\Delta_k$ is finite, ${\mathcal Q}(g)$ is finite.

\medskip
{\it Proof of the Claim} Let $g = (v,b) \in \Delta_k$ and $t \in T$. Then $Htg  = H \tilde{t}$ for some $\widetilde{t} \in T$.
Then the cofactor is $tg \widetilde{t}^{-1} = tvb \widetilde{t}^{-1} = (tvb(\widetilde{t})^{-1}, b) \in V_0 \rtimes B$. Applying the virtual endomorphism $f$ we obtain the element
$$
g_t = (A(tvb(\widetilde{t})^{-1}), Ab A^{-1}) \in V \rtimes B.
$$
We claim that $g_t \in \Delta_k$. It remains to show that $\rho(A(tvb(\widetilde{t})^{-1})) \leq k$. Note that  by the definition of $A$ for any non-zero matrix column $w$ we have
$
\rho(A(w)) \leq \rho(w)$, in particular
$$
\rho(A(tvb(\widetilde{t})^{-1}) ) \leq \rho(tvb(\widetilde{t})^{-1}). 
$$
Furthermore
$$
\rho(tvb(\widetilde{t})^{-1}) \leq max \{ \rho(t), \rho(v), \rho(b(\widetilde{t})^{-1}) \} $$
and $\rho(t) = 0, \rho(v) \leq k$. Finally
$$
\rho(b(\widetilde{t})^{-1})  = \rho(b(\widetilde{t})) \leq \rho(b) \leq k.
$$

\end{proof}

\section{Explicit calculations of the states in the metabelian case } \label{Construction}

In this section  we consider in more details  the metabelian case of Theorem A i.e. $m = 2$. Here we denote by $Q$ the group that was denoted by $Q_0$ in section \ref{S-arithmetic} and note that we actually consider a subgroup of finite index of the original group from Theorem A.  All actions and modules are right ones and as before $p$ is a prime number.
Let  $f_0 = x, f_1, \ldots, f_{n-1} \in  {\mathbb{F}}_p[x] \setminus \mathbb{F}_p(x-1)$ be pair-wise different, irreducible, monic polynomials. To simplify the calculations we assume further that
\begin{equation} \label{1-condition}
f_i(1) = 1 \hbox{ for } 1 \leq i \leq n-1.
\end{equation}
Consider the group
$$
G = A \rtimes Q, \hbox{ where } A = {\mathbb{F}}_p[x^{\pm 1}, {1 \over {f_1}}, \ldots, {1 \over {f_{n-1}}} ] \hbox{ and } Q = \langle x_0, \ldots, x_{n-1} \rangle \simeq {\mathbb{Z}}^n,
$$
where  $x_i$ acts on $A$ via multiplication with $f_i$ for $0 \leq i \leq n-1$.

The homological finiteness properties in this case follow from earlier results in  \cite{Bux} , \cite{Desi} and do not involve the theory of buildings.

 \begin{theorem} \label{metabelian0} For $n \geq 2$, 
 $G$ is finitely presented, of type $FP_n$ but not of type $FP_{n+1}$.
 \end{theorem}
 
  \begin{proof} 
 By Theorem \ref{cond-met} and Theorem  \ref{fpm-met}, Theorem \ref{metabelian0}  is equivalent to $A$ is $n$-tame but not $(n+1)$-tame as $\mathbb{Z} Q$-module.
  In our case , because of the link between $\Sigma^c$ and valuation-theory described in \cite{BG}, $\Sigma^c_A(Q) = \{ [\chi_i] \}_{0 \leq i \leq n}$, where 
  
  a) $\chi_i(x_j) = \delta_{i,j}$ for $0 \leq i,j \leq n-1$, where $\delta_{i,j}$ is the Kronecker symbol;
  
  b) $\chi_n(x_j) = - deg (f_j)$ for $0 \leq j \leq n-1$.
    
Thus every $n$ elements of $\Sigma_A^c(Q)$ lie in an open semi-sphere of $S(Q)$ but the whole set $\Sigma_A^c(Q)$ does not lie in an open semi-sphere of $S(Q)$, so $A$ is $n$-tame but not $(n+1)$-tame as $\mathbb{Z} Q$-module.    
  \end{proof}
\noindent
We move from additive to multiplicative notation.
 Let $
\dot{A}$ be the multiplicative group formed by $u^{r}$ where $r\in A$ and 
$$
u^{p} =1 \hbox{ and }
u^{r_{1}}u^{r_{2}} =u^{r_{1}+r_{2}}\text{.}
$$
Thus $G \simeq \dot{A}\rtimes Q$, where $Q$ acts on $\dot{A}$
as 
\[
x_{i}:u^{r}\rightarrow u^{r.x_{i}}\hbox{ for } 0\leq i\leq n-1. 
\]
From now on we identify $G$ with  $\dot{A}\rtimes Q$.
Note that  $Q$ acts faithfully on $\dot{A}$ and $G$ is a metabelian group
generated by $\left\{ u,x_{0},x_{1},...,x_{n-1}\right\} $.

 Consider the ideal $ I = (x - 1)A$ of $A$.  Recall that $A=I+\mathbb{F}_{p}$.
Let $\dot{I}$ be the corresponding subgroup of $\dot{A}$ formed by elements $
u^{r}$ such that $r\in I$. Define the subgroup $H$ of $G$ 
\[
H=\dot{I}\rtimes Q\text{.} 
\]
Then $H$ is a subgroup of $G\,$\ of index $p$ with transversal $
T=\left\{ e,u,u^{2},...,u^{p-1}\right\} $ which we will re-write later as $
\left\{ 0,1,...,p-1\right\} $. Elements of $G$ have the unique form $u^{r}q$
and elements of $H$ have the unique form $u^{\left( x-1\right) r}q$, where $
r\in A,q\in Q$. By (\ref{1-condition}) it is easy to deduce that $H$ is normal in $G$.

Define the virtual endomorphism 
\begin{eqnarray*}
f &:&H\rightarrow G,\text{ } \\
u^{\left( x-1\right) r}q &\rightarrow &u^{r}q\text{ for all }r\in A\text{.}
\end{eqnarray*}

In the rest of the section we make explicit calculation of the states of the generators of $G$. Write $e$ for the identity element.

\textbf{The permutational representation of }$\sigma :G\rightarrow
Perm\left( \left\{ 0,1,...,p-1\right\} \right) $ \textbf{on the transversal }
$T$.

\medskip
\noindent
(i) $u:Hu^{i}\rightarrow Hu^{i+1}$ for all $0\leq i\leq p-2$ and $
Hu^{p-1}\rightarrow Hu^{p}=He=H$; therefore, 
\[
\sigma \left( u\right) =\left( 0,1,...,p-1\right)  \in S_p; 
\]
(ii) $x_{j}:Hu^{i}\rightarrow Hu^{i}x_{j}=H\left( x_{j}\right)
^{u^{-i}}u^{i}=Hu^{i}$; therefore, 
\[
\sigma \left( x_{j}\right) = e.
\]

\textbf{The vector of co-factors}\newline
(1) of $u$ is
\[
\left( e\right) _{0 \leq i \leq p-1}; 
\]
(2) of $x_{j}$ is \[
\left( u^{i\left( 1-f_{j}^{-1}\right) }x_{j}\right) _{0 \leq i \leq p-1}, 
\] since $u^{i}x_{j}u^{-i}=\left( u^{i}\left( u^{-i}\right)
^{x_{j}^{-1}}\right) x_{j}= u^{i\left( 1-f_{j}^{-1}\right) }x_{j}. $

We apply $f$ to the entries of the above vectors to obtain
\textbf{the image of the generators of }$G$ \textbf{ in } $Aut({\mathcal A}_p)$:
\begin{equation} \label{pintura1}
u=\left( 0,1,...,p-1\right) \in S_p ; 
\end{equation}
\begin{equation} \label{pintura2} 
x_{j} =\left( u^{-i\left( \frac{f_{j}^{-1}-1}{x-1}\right) }x_{j}\right)
_{0 \leq i \leq p-1}, 
\end{equation}
\begin{equation} \label{pintura300}
x_{j}^{-1} =\left( x_{j}^{-1}u^{i\left( \frac{f_{j}^{-1}-1}{x-1}\right)
}\right) _{0 \leq i \leq p-1}=\left( u^{-i\left( \frac{f_{j}-1}{x-1}\right)
}x_{j}^{-1}\right) _{0 \leq i \leq p-1}\text{.}
\end{equation}
Thus
$$
	x_0^{-1} = \left( u^{-i}
x_{0}^{-1}\right) _{0 \leq i \leq p-1}.
$$
Note that for $p=2$,
$
x_{0}^{-1}=\left( x_{0}^{-1},u^{-1}x_{0}^{-1}\right) 
$
is just the classical representation of the lamplighter group generated by $x_0$ and $u$.

For $g \in G$ we write $(g)^{(1)}$ for $(g_0, \ldots, g_{p-1})$, where $g_0 = g_1 = \ldots = g_{p-1} = g$.

\begin{lemma} \label{power} For every positive integer $i$ we have
$$
u^{x^{i}} = (u^{x^{i-1}} u^{x^{i-2}} \ldots u^x u)^{(1)} u. 
$$
Hence for $\lambda \in \mathbb{F}_p[x]$ we have
$$
u^{\lambda} = ( u^{{\lambda - \lambda(1)} \over {x -1} })^{(1)} u^{\lambda(1)}.
$$
\end{lemma}

\begin{proof} To prove the first equality we induct on $i$.
We calculate
 first $u^x = x_0^{-1} u x_0 $. Using (\ref{product}), (\ref{pintura1}), (\ref{pintura2}) and (\ref{pintura300}) we obtain
\noindent
\begin{eqnarray*}
\text{ }u^{x} &&=\left( x_{0}^{-1},\text{ }u^{-1} x_{0}^{-1},\ldots,\text{ }u^{1-p}x_{0}^{-1}\right) \left(0, 1,2, \ldots, p-1\right) \left( x_{0},\text{ }x_{0} u ,\ldots,\text{ }x_0 u^{p-1} \right) \\
&&= \left( x_{0}^{-1},\text{ }u^{-1} x_{0}^{-1},\ldots,\text{ }u^{1-p}x_{0}^{-1}\right) \left( \text{ }x_{0} u ,\ldots,\text{ }x_0 u^{p-1}, x_0 \right) \left(0, 1,2, \ldots, p-1\right) \\
&&= (u, \ldots, u) u = (u)^{(1)} u.
\end{eqnarray*}

For the inductive step we have
 \begin{eqnarray*}
\text{ }u^{x^i} &&=\left( x_{0}^{-1},\text{ }u^{-1} x_{0}^{-1},\ldots,\text{ }u^{1-p}x_{0}^{-1}\right) u^{x^{i-1}} \left( x_{0},\text{ }x_{0} u ,\ldots,\text{ }x_0 u^{p-1} \right) \\
&&=\left( x_{0}^{-1},\text{ }u^{-1} x_{0}^{-1},\ldots,\text{ }u^{1-p}x_{0}^{-1}\right) (u^{x^{i-2}} \ldots u^x u)^{(1)} u \left( x_{0},\text{ }x_{0} u ,\ldots,\text{ }x_0 u^{p-1} \right) \\
&&=\left( x_{0}^{-1},\text{ }u^{-1} x_{0}^{-1},\ldots,\text{ }u^{1-p}x_{0}^{-1}\right) (u^{x^{i-2}} \ldots u^x u)^{(1)}  \left( \text{ }x_{0} u ,\ldots,\text{ }x_0 u^{p-1}, x_0 \right) u \\
&&= (u^{x^{i-1}} \ldots u^x u, \ldots, u^{x^{i-1}} \ldots u^x u) u = (u^{x^{i-1}} \ldots u^x u)^{(1)} u
\end{eqnarray*}
To prove the second equality we induct on $deg(\lambda)$ if $\lambda \not= 0$. The case $deg(\lambda) = 0$ is obvious. Suppose that $deg(\lambda) = m > 0$, $\lambda = h + \alpha x^m$ for $h \in \mathbb{F}_p[x]$, $deg(h) < m$ and $\alpha \in {\mathbb{F}}_p \setminus \{ 0 \}$. Then by induction
 \begin{eqnarray*}
\text{ }u^{\lambda} &&=u^h  (u^{x^m})^{\alpha}  =  ( u^{{h - h(1)} \over {x -1} })^{(1)} u^{h(1)} \left( ( u^{{x^m - 1} \over {x -1} })^{(1)} u\right)^{\alpha}\\
&&= ( u^{{h - h(1)} \over {x -1} }  u^{\alpha {{x^m - 1} \over {x -1} }})^{(1)} u^{h(1)+ \alpha} = ( u^{{\lambda - \lambda(1)} \over {x -1} })^{(1)} u^{\lambda(1)}.  \\
\end{eqnarray*}
\end{proof}

\begin{lemma} \label{degree} Let $j \in \{ 0,1,2, \ldots, n-1 \}$.
The subset $$Y_j = \{   u^{\lambda} x_j^{-1} \mid \lambda \in \mathbb{F}_p[x], deg (\lambda) \leq deg(f_j) \hbox{ or } \lambda = 0 \}$$ of $G$ is state closed i.e. for ever $g \in Y_j$ we have a recursive presentation $g = (g_0, g_1, g_2, \ldots, g_{p-1}) \pi$, where $\pi = \sigma(g)$ is a permutation and $g_0,g_1, g_2, \ldots,g_{p-1} \in Y_j$. In particular $G$ is a transitive automata group  i.e. a finite state, transitive self-similar group.
\end{lemma}

\begin{proof} Assume that $\lambda \not= 0$ and $deg(\lambda) \leq deg(f_j)$. By the Euclidian algorithm 
$
\lambda = (x-1) \widetilde{\lambda} + \lambda(1), $ where  $deg(\widetilde{\lambda} ) < deg (\lambda)$ if $\widetilde{\lambda} \not= 0$. 
Recall that by Lemma \ref{power}, (\ref{pintura2}) and (\ref{pintura300}) 
$$
u^{\lambda} = ( u^{{\lambda - \lambda(1)} \over {x -1} })^{(1)} u^{\lambda(1)}
\hbox{  and  }
x_{j}^{-1} =\left( u^{-i\left( \frac{f_{j}-1}{x-1}\right)
}x_{j}^{-1}\right) _{0 \leq i \leq p-1}\text{.}$$
Then since $u = (0,1,2, \ldots, p-1)$
\begin{eqnarray*}
u^{\lambda} x_{j}^{-1} &&=(u^{\widetilde{\lambda}})^{(1)} u^{\lambda(1)} \left( u^{-i\left( \frac{f_{j}-1}{x-1}\right)
}x_{j}^{-1}\right) _{0 \leq i \leq p-1}\\
&&= (u^{\widetilde{\lambda}})^{(1)}  \left( u^{- \left((i)u^{\lambda(1)} \right)\left( \frac{f_{j}-1}{x-1}\right)
}x_{j}^{-1}\right) _{0 \leq i \leq p-1}  u^{\lambda(1)}
\\ && =  \left( u^{\widetilde{\lambda} - \left((i)u^{\lambda(1)} \right)\left( \frac{f_{j}-1}{x-1}\right)
}x_{j}^{-1}\right) _{0 \leq i \leq p-1}  u^{\lambda(1)} \\
\end{eqnarray*}%
has states $\{ u^{\widetilde{\lambda} - \left((i) u^{\lambda(1)} \right) \left( \frac{f_{j}-1}{x-1}\right)
}x_{j}^{-1} \}_{0 \leq i \leq p-1} $. Note that $\widetilde{\lambda} - \left( (i) u^{\lambda(1)}\right)\left( \frac{f_{j}-1}{x-1}\right)
$ is a polynomial in $\mathbb{F}_p[x]$ of degree smaller than the degree of $f_j$.
\end{proof}

\section{Commutative algebra approach} \label{Krull-dim}

In this section we show that the metabelian example considered in Section  \ref{Construction} generalizes significantly. What makes the example in Section \ref{Construction} work is that $A$ is of Krull dimension 1. 
We observe that similar results will not hold if A has Krull dimension 2. Recently Sidki and Dantas showed that the metabelian group $A \rtimes Q = \mathbb{Z} \wr \mathbb{Z}$ is not transitive self-similar  \cite{Alex-Said2}, note that in this example $A$ is a domain of Krull dimension 2.

\begin{lemma} \label{Krull1domain} Let $Q$ be a finitely generated abelian group and $P$ be a prime ideal in $\mathbb{Z} Q$ such that  $B = \mathbb{Z} Q / P$ has Krull dimension 1.
 Then every non-zero ideal $I$ of $B$ has finite index i.e. $B/ I$ is a finite ring.
\end{lemma}
\begin{proof} Indeed consider the radical ideal of $I$ defined by 
    $$
    \sqrt{I} = \{ r \in B \mid \hbox{ there is } k \geq 1 \hbox{ such that } r^k \in I \}.
    $$
    There are only finitely many minimal prime ideals of $B$ above any fixed ideal, in particular above $I$. Let $P_1, \ldots, P_k$ be the minimal prime ideals above the ideal $I$. Then
    $$
    \sqrt{I} = P_1 \cap P_2 \cap \ldots \cap P_k.
    $$
    Since $B$ has Krull dimension 1 and $B$ is a domain we deduce $0 \subset P_i$ is  maximal  chain of prime ideals, hence  each $P_i$ is a maximal ideal in $B$. Since every maximal ideal in $\mathbb{Z} Q$ has finite index ( actually this holds in more general situation : simple  modules over group algebra of polycyclic groups are finite \cite{Jim}), we obtain that $B/ P_i$ is finite for every $i$. Then  $B / \sqrt{I}$ embeds in the finite ring $B/ P_1 \times B/ P_2 \times \ldots \times B/ P_k$, in particular
    \begin{equation} \label{finite}  B/ \sqrt{I} \hbox{ is finite}. \end{equation}
   Since $B$ is a Noetherian ring $\sqrt{I}$ is finitely generated as an ideal, say by $b_1, \ldots, b_m$. Then there is a positive integer $k$ such that $b_i^k \in \sqrt{I}$ for all $ 1 \leq i \leq m$, hence 
   \begin{equation} \label{inclusion}
   \sqrt{I}^s \subseteq I \hbox{ for } s = m(k-1) + 1.
   \end{equation}
   Observe that there is a filtration of ideals in $B$ 
   $$
   \sqrt{I}^s \subseteq \sqrt{I}^{s-1} \subseteq \ldots \subseteq \sqrt{I}^{i} \subseteq \sqrt{I}^{i-1} \subseteq \ldots \subseteq \sqrt{I}.
   $$ 
By Noetherianess each quotient $\sqrt{I}^{i-1}/ \sqrt{I}^i$ is a finitely generated $B$-module, where $\sqrt{I}$ acts as 0 i.e. is a finitely generated $B/ \sqrt{I}$-module. This together with (\ref{finite}) implies that  each quotient $\sqrt{I}^{i-1}/ \sqrt{I}^i$ is finite. Hence 
\begin{equation} \label{finite2}
\sqrt{I} /\sqrt{I}^s \hbox{ is finite.}
\end{equation}
By (\ref{finite}) and (\ref{finite2}) $B/ \sqrt{I}^s$ is finite. This together with (\ref{inclusion}) implies that $B/ I$ is finite.
   \end{proof}

\begin{theorem} \label{thm-general} Let $Q$ be a finitely generated abelian group and  $B$ be a finitely generated $\mathbb{Z} Q$-module such that $C_Q(B) = \{ q \in Q \mid B(q-1) = 0 \} = 1$. Let $P_1, \ldots, P_s$ be prime ideal  of $\mathbb{Z} Q$ such that 

1. there is a filtration of $\mathbb{Z} Q$-submodules of $B$
$$
B_0 = 0 \subset B_1 \subset B_2 \subset \ldots \subset B_i \subset B_{i+1} \subset \ldots \subset B_{r-1} \subset B_r = B,
$$
where $B_{i} / B_{i-1} \simeq \mathbb{Z} Q/ P_{\alpha_i}$ for  $1 \leq i \leq r$, where $\{ \alpha_1, \ldots, \alpha_{r} \} = \{ 1, 2, \ldots, s \}$;

2. there is some $d \in \{ 1, \ldots, r \}$  such that the domain $\mathbb{Z} Q/ P_i$ has Krull dimension  1 for every $1 \leq i \leq d$ and if $d < r$ the domain $\mathbb{Z} Q/ P_i$ has Krull dimension  0 (i.e. is finite) for every $d+1 \leq i \leq s$;

3. there is an element $\delta \in \mathbb{Z} Q \setminus (\cup_{ 1 \leq i \leq d} P_i)$ such that   the image  of $\delta$ in $\mathbb{Z} Q/P_i$ is not invertible for every $ 1 \leq i \leq d$.
  
 Then  $G = B \rtimes Q$ is a transitive self-similar group.
\end{theorem}

\begin{proof}
 We will prove the existence of a homomorphism of $\mathbb{Z} Q$-modules
 \begin{equation} \label{f-hom}
 f : \widetilde{B} \to B
  \end{equation}
  such that $\widetilde{B}$ is a $\mathbb{Z} Q$-submodule of $B$ of finite index, i.e. $B/ \widetilde{B}$ is finite, and there is not a non-trivial  $\mathbb{Z} Q$-submodule $M$ of $\widetilde{B}$ such that $f(M) \subseteq M$. Then we set $H = \widetilde{B} \rtimes Q$ and extend $f$ to a virtual endomorphism of groups
  \begin{equation} \label{f-general}
  f : H \to G
   \end{equation}
   by defining $f\mid_Q = id_Q$.  Then it would follow that $f$ is a simple virtual endomorphism of $G$. For, suppose
   that $K$ is a normal subgroup of $G$ such that $f(K) \subseteq K$. Then $M = K \cap B = K \cap \widetilde{B}$ is a $\mathbb{Z} Q$-submodule of  $\widetilde{B}$ such that $f(M) \subseteq M$, hence $M = 0$.
   Then $B K \simeq B \times K$, hence $K$ acts trivially on $B$ via conjugation. Then for the canonical map $\pi : G \to G/ B = Q$ we have that $\pi(K) \subseteq C_{Q} (B) = \{ q \in Q \mid $ q acts trivially on $B \} = 1$.  Since $B \cap K = 0$ we deduce that $K \simeq \pi(K) = 1$, hence $f$ from (\ref{f-general})  is a simple endomorphism. Thus $G$ is a transitive self-similar group.

Note that for some $b_1, \ldots, b_r \in B$ we have
$$
B_i = b_1 \mathbb{Z} Q  + \ldots + b_i \mathbb{Z} Q \hbox{ for } 1 \leq i \leq r.
$$
 We claim that there is a $\mathbb{Z} Q$-submodule  $C$ of finite index in $B$ such that  if  $ B_{i} / B_{i-1} \simeq \mathbb{Z} Q/ P_{\alpha_i}$ is finite for some $ 1 \leq i \leq r$ then $C \cap B_{i} = C \cap B_{i-1}$. Since $C$ has finite index in $B$  whenever $B_{i}/ B_{i-1}$ is infinite, the quotient $(B_{i} \cap C)/ (B_{i-1} \cap C)$ is an ideal of finite index in 
 $B_{i}/ B_{i-1}$. To prove the existence of $C$ we induct on $r$, the case $ r = 1$ is obvious. By the inductive hypothesis there is a $\mathbb{Z} Q$-submodule $\widetilde{C}$ of $B_{r-1}$ such that if  $ B_{i} / B_{i-1} \simeq \mathbb{Z} Q/ P_{\alpha_i}$ is finite  for some $ 1 \leq i \leq r-1$ then $\widetilde{C} \cap B_{i} = \widetilde{C} \cap B_{i-1}$. If $B_r/ B_{r-1}$ is finite then we set $C = \widetilde{C}$. Now suppose that $B_r/ B_{r-1}$ is infinite. It suffices to find a $\mathbb{Z} Q$-submodule $C$ of finite index in $B$ such that $C \cap B_{r-1} \subseteq \widetilde{C}$. Consider the group $\widetilde{G} = (B/ \widetilde{C}) \rtimes Q$. Since every finitely generated  metabelian group is residually finite (actually by \cite{Hall} every finitely generated abelian-by-nilpotent group is residually finite) and $B_{r-1}/ \widetilde{C}$ is finite there is an epimorphism
 $$
 \theta : \widetilde{G} \to \Delta
 $$
 where $\Delta$ is a finite group  such that $Ker (\theta) \cap (B_{r-1} / \widetilde{C}) = 1$. Then $V = Ker (\theta) \cap (B / \widetilde{C})$ is a normal subgroup of $\widetilde{G}$ inside $B/ \widetilde{C}$, hence $V$ is a $\mathbb{Z} Q$-submodule of $B/ \widetilde{C}$.  Finally we define $C$ as the preimage of $V$ in $B$ under the canonical epimorphism $B \to B/ \widetilde{C}$.

Consider the virtual endomorphism
$$
f : \widetilde{B} =  C \delta \to B
$$
given by $f( c \delta ) = c$ for $c \in C$. Observe that $f$ is a homomorphism of $\mathbb{Z} Q$-module and $f$ is well defined since $\delta$ is not a zero divisor in every quotient $(C \cap B_i)/ (C \cap B_{i-1})$ for $ 1 \leq i \leq r$, hence $\delta$ is not a zero divisor in $C$. Indeed either $(C \cap B_i)/ (C \cap B_{i-1}) = 0$ or $(C \cap B_i)/ (C \cap B_{i-1})$ has finite index in $B_i/ B_{i-1}$ and $B_i/ B_{i-1}$ has Krull dimension 1. By assumption $\delta$ is not a zero divisor in $B_i/ B_{i-1}$ if $B_i / B_{i-1}$ has Krull dimension 1.

Observe that since  $(C \cap B_i)/ (C \cap B_{i-1}) = 0$ or $(C \cap B_i)/ (C \cap B_{i-1})$ has finite index in $B_i/ B_{i-1}$  and $B_i/ B_{i-1}$ has Krull dimension 1, there is a filtration of $C$ with quotients $M_j$ cyclic $\mathbb{Z} Q$-modules, each of Krull dimension at most 1. Then
by Lemma \ref{Krull1domain}   $M_{j}/ M_{j} \delta$ is finite for each $j$ and so $C/ C \delta$ is finite. Since $B/C$ is finite we deduce that $\widetilde{B} = C \delta$ has finite index in $B$.

Suppose that $M$ is a $\mathbb{Z} Q$-submodule of $\widetilde{B}$ such that $f(M) \subseteq M$. Then $M \subseteq \cap_{j \geq 1}  C \delta^j $ and
 $(M \cap B_i)/ (M \cap B_{i-1}) = (M \cap B_i \cap C)/ (M \cap B_{i-1} \cap C)$ is a $\mathbb{Z} Q$-submodule of $(B_i \cap C)/(B_{i-1} \cap C)$.
If $B_i/B_{i-1}$ is finite, we deduce that  $(B_i \cap C)/(B_{i-1} \cap C) =0$ and so $M \cap B_i = M \cap B_{i-1}$.

 From now on suppose that $B_i/ B_{i-1}$ is infinite, so it is a domain of Krull dimension 1.
We claim  that for the image $M_i$  of $M\cap B_i$ in $V_i = B_i  / B_{i-1}$ we have $$M_i \subseteq \cap_{j \geq 1}  V_i \delta^j.$$ Indeed to prove this observe that $\delta$ is not a zero divisor of $C / (C \cap B_i)$ since $C / (C \cap B_i)$ has filtration with quotients $(C \cap B_j)/ (C \cap B_{j-1})$ for $j \geq i+1$ that are either 0 or are finite index ideals in $V_j$ and $V_j$ has Krull dimension 1 and in this case by assumption $\delta$ is not a zero divisor. Now since $\delta$ is not a zero divisor of $C / (C \cap B_i)$ we have that $C \delta \cap B_i = (C \cap B_i) \delta$ and by induction on $j$ we have $C \delta^j \cap B_i = (C \cap B_i) \delta^j$. Thus
$$M \cap B_i \subseteq B_i \cap (\cap_{j \geq 1} C \delta^j) \subseteq \cap_{j \geq 1} (C \cap B_i) \delta^j \subseteq \cap_{j \geq 1} B_i \delta^ j.$$

Since the image of $\delta$ is not invertible in $V_i$ we get that $V_i \delta$ is a proper ideal of $V_i$.
Let $\widetilde{Q}_i$ be a maximal ideal of $V_i$ above $V_i \delta$. 
 Consider the multiplicatively closed set $S_i =  V_i \setminus \widetilde{Q}_i$ and the localization $V_i S_i^{-1}$ i.e. $V_i S_i^{-1}$ is a subring  of the field of fractions of $V_i$. Then $Q_{0,i} = \widetilde{Q}_i S_i^{-1}$ is the unique maximal ideal of $A_i = V_i S_i^{-1}$ i.e. $A_i$ is a local ring. Then the Jacobson radical $J(A_i) = Q_{0,i}$ and for $I_i = \cap_{j \geq 1} Q_{0,i}^j$ it is easy to see that $I_i = 0$, hence $\cap_{j \geq 1}  V_i \delta^j = 0$ and consequently  $0 = M_i = (M \cap B_i)/ (M \cap B_{i-1})$. Indeed if $I_i \not= 0$ then its radical $\sqrt{I_i}$ is the intersection of all primes ideals above it. Since $A_i$ has Krull dimension 1 any prime ideal above $\sqrt{I_i}$ is a maximal one, hence is $Q_{0,i}$ and $\sqrt{I_i} = Q_{0,i}$. By Noetherianess of $A_i$ there is an integer $k$ such that $Q_{0,i}^k = \sqrt{I_i}^k \subseteq I_i = \cap_{j \geq 1} Q_{0,i}^j$. Thus $Q_{0,i}^k = Q_{0,i}^{k+1} = I_i = I_i Q_{0,i}$ and by
  Nakayama lemma $I_i = 0$, a contradiction.

Note that we have proved that $M \cap B_j = M \cap B_{j-1}$ for every $1 \leq j \leq r$, hence $M  = 0$ and $f$ is a simple endomorphism.
\end{proof}

\begin{theorem}  \label{conditions-Krull} Let $Q$ be a finitely generated abelian group and  $B$ be a finitely generated $\mathbb{Z} Q$-module of Krull dimension 1  such that $C_Q(B) = \{ q \in Q \mid B(q-1) = 0 \} = 1$. Then $G = B \rtimes Q$ is a transitive self-similar group.
\end{theorem}
\begin{proof} By the preliminaries on commutative algebra, see Section \ref{comm} , 
there are prime ideals $P_1, \ldots, P_s$ of $\mathbb{Z} Q$ such that 
there is a filtration of $\mathbb{Z} Q$-submodules of $B$
$$
B_0 = 0 \subset B_1 \subset B_2 \subset \ldots \subset B_i \subset B_{i+1} \subset \ldots \subset B_{r-1} \subset B_r = B,
$$
where $B_{i} / B_{i-1} \simeq \mathbb{Z} Q/ P_{\alpha_i}$ for  $1 \leq i \leq r$, $\{ \alpha_1, \ldots, \alpha_{r} \} = \{ 1, 2, \ldots, s \}$ and  furthermore
\begin{equation} \label{ass}
Ass(B) \subseteq \{ P_1, \ldots, P_{s} \} \subseteq Supp (B),
\end{equation}
where the minimal elements of  the three sets of primes considered above are the same. Since $B$ has Krull dimension 1, each quotient $B_i/ B_{i-1} \simeq \mathbb{Z} Q/ P_{\alpha_i}$ has Krull dimension at most 1 i.e. either has Krull dimension 1 or has Krull dimension 0, in this case it is finite.

   Let $P_1, \ldots, P_d$ be the ideals among the elements of the set of prime ideals $\{ P_1, \ldots, P_s \}$ such that $\mathbb{Z} Q/ P_i$ has Krull dimension 1.
      Let $Q_i$ be a maximal ideal of $\mathbb{Z} Q$ such that ${P}_i \subset Q_i$ for $ 1 \leq i \leq d$. We claim that
      there is an element 
      \begin{equation} \label{delta1} 
      \delta \in \mathbb{Z} Q \setminus (\cup_{ 1 \leq i \leq d} {P}_i)
      \end{equation}  such that for every $ 1 \leq i \leq d$  the image  of $\delta$ in $\mathbb{Z} Q/{P}_i$ is not invertible. It suffices to choose 
      \begin{equation} \label{delta2}
      \delta \in \cap_{1 \leq i \leq d} Q_i
      \end{equation}
       to guarantee that 
         the image  of $\delta$ in $\mathbb{Z} Q/{P}_i$ is not invertible for $ 1 \leq i \leq d$. This will be impossible if
         \begin{equation} \label{max-prime}
          \cap_{1 \leq i \leq d} Q_i \subseteq 
          \cup_{ 1 \leq i \leq d} {P}_i.
          \end{equation}
					Suppose from now on that (\ref{max-prime}) holds.
 We can choose $Q_i \not= Q_j$ for $i \not= j$. Then for $i \not= j$ we have that  $Q_i$ and $Q_j$ are coprime ideals in the sense that $Q_i + Q_j = \mathbb{Z} Q$, hence $  \cap_{1 \leq i \leq d} Q_i = \prod_{1 \leq i \leq d} Q_i$ and so
 $$
 \prod_{1 \leq i \leq d} Q_i \subseteq 
 \cup_{1 \leq i \leq d} {P}_i.
 $$
 Assume that $j$ is the minimal number such that for some $1 \leq i_1 < \ldots < i_j \leq d$ we have
 \begin{equation} \label{pintura3}
 Q_{i_1} \ldots Q_{i_j} \subseteq  
  \cup_{1 \leq i \leq d} {P}_i.
 \end{equation}
 If $j \geq 2$ then there are elements $q_2 \in Q_{i_2}, \ldots, q_j \in Q_{i_j}$ such that $q_2 \ldots q_j \notin   \cup_{1 \leq i \leq d} {P}_i.$ Since
 $Q_{i_1} q_2 \ldots q_j \in  Q_{i_1} \ldots Q_{i_j} \subseteq    \cup_{1 \leq i \leq d} {P}_i$ and each $P_i$ is a prime ideal  we deduce that for $\widetilde{Q} = Q_{i_1}$  we have
 \begin{equation} \label{pintura4}
\widetilde{ Q}  \subseteq   \cup_{1 \leq i \leq d} {P}_i,
 \end{equation}
 a contradiction with the minimality of $j$. Hence $j = 1$
  and  (\ref{pintura4}) still  holds. 
  
  \noindent
 Let $J_i = \widetilde{Q} \cap P_i$, hence
 $$
 \widetilde{Q} = \cup_{1 \leq i \leq d} J_i.
 $$ 
 Choose $$q_j \in (\cap_{1 \leq i \not= j \leq d} J_i) \setminus J_j.$$ 
 Then 
 $$
 q = q_1 + \ldots + q_{d} \in \widetilde{Q} = J_1 \cup \ldots \cup J_{d}
 $$
 so for some $i$ we have  $q \in J_i$. Hence
 $
 q_i = q - (\sum_{j \not= i} q_j) \in J_i
 $, a contradiction. Then for some $ 1 \leq j \leq d$ we have $\cap_{1 \leq i \not= j \leq d} J_i \subseteq J_j$, 
hence $$\prod_{1 \leq i \not= j \leq d} J_i  \subseteq \cap_{1 \leq i \not= j \leq d} J_i  \subseteq J_j \subseteq P_j.$$ Since $P_j$ 
is a prime ideal, for some $i \not= j$ we have $$\widetilde{Q} P_i \subseteq \widetilde{Q} \cap P_i = J_i \subseteq P_j.$$ Using again that $P_j$ is a prime ideal and $P_i \subsetneq P_j$ for $1 \leq i \not= j \leq d$ we deduce that
 $$
\widetilde{ Q} \subseteq {P}_j.
 $$
 Thus the infinite ring (of Krull dimension 1) $\mathbb{Z} Q/ {P}_j$ is a quotient of the finite ring (of Krull dimension 0)  $\mathbb{Z} Q/ \widetilde{Q}$, a contradiction. 

 Then there is an element $\delta \in \mathbb{Z} Q$ such that
 (\ref{delta1}) and (\ref{delta2}) hold and we can apply the previous theorem.
 \end{proof}

\begin{cor} Let $G = B \rtimes Q$ be a  constructible group, where $B$ and $Q$ are abelian such that $C_Q(B) = 1$. Then $G$ is a transitive self-similar group.
\end{cor}
\begin{proof}
By definition constructible soluble groups are obtained from the trivial group using ascending HNN extensions and finite extensions, where the base and associated subgroups are already constructible. Thus  $G$ is of finite Pr\"ufer rank, hence  the Krull dimension of $B$ is 1.  Then we can apply Theorem \ref{conditions-Krull}.
\end{proof}

\section{A localised version of Dantas-Sidki example} \label{local}

In \cite{Alex-Said} Dantas and Sidki considered the following metabelian group
$$
G_{p,d} = C_p \wr \mathbb{Z}^d = \langle a \rangle \wr \langle x_1, \ldots, x_d \rangle,
$$
where $p$ is a prime number and $C_p$ is the cyclic group of order $p$.
Note that $G_{p,d}$ is finitely generated but not finitely presented.
The normal closure $A$ of $a$ in $G_{p,d}$ is $\mathbb{F}_p[Q]$, where $Q = \langle x_1, \ldots, x_d \rangle \simeq \mathbb{Z}^d$. Set
$$Q_0 = \langle x_1^p, x_2, \ldots, x_d \rangle \hbox{ and } H = G' Q_0,$$
where $G'$ is the commutator subgroup of $G$.
Thus $A_0 = A \cap H \simeq Aug(\mathbb{F}_p Q)$ is the augmentation ideal of the group algebra and
$$H = A_0 \rtimes Q_0.$$  Furthermore as shown in \cite{Alex-Said} there is a simple virtual endomorphism
$$
f : H \to G
$$ given by
$$ f(a^{x_1^i - 1}) = a^{i} \hbox{ for } 1 \leq i \leq p-1 \hbox{ and } f(a^{z-1}) = 1 \hbox{ for } z \in Q_0 $$
and
$$
f(x_1^p) = x_2, f(x_i) = x_{i+1} \hbox{ for } 2 \leq i \leq d-1, f(x_d) = x_1.
$$ 
In this section we show how localizing the above example we can construct transitive self-similar metabelian group $ \widetilde{A} \rtimes \widetilde{Q}$ that is finitely presented. In this example the Krull dimension of $\widetilde{A}$ is precisely $d$, thus showing that there are finitely presented examples where the Krull dimension is not 1.

Let $g \in \mathbb{F}_p[x] \setminus (\cup_{j \geq 0} \mathbb{F}_p x^j \cup (x-1)\mathbb{F}_p[x])$. Let 
$$
s_i = g(x_i) \in \mathbb{F}_p[x_i].
$$
Then
$$ 
S = \{ s_1^{z_1} \ldots s_d^{z_d} \mid z_i \geq 0, 1 \leq i \leq d \} \subseteq \mathbb{F}_p[x_1, \ldots, x_d]
$$
is a multiplicatively closed subset of $A \setminus \{ 0 \}$.
Define
$$\widetilde{G} = \widetilde{A} \rtimes \widetilde{Q},
\hbox{  where }\widetilde{A}\hbox{  is the localization }A S^{-1},$$  we view $A$ as right $\mathbb{F}_p Q$-module with $Q$ action (via conjugation) given by multiplication, $$\widetilde{Q} = Q \times Q_1, \hbox{ where }Q_1 = \langle y_1, \ldots, y_d \rangle$$ and each $y_i$ acts via conjugation (on the right) on $A S^{-1}$ by multiplication with $s_i \in S$.
Note that
$$
\widetilde{A} = A S^{-1} = \mathbb{F}_p[x_1^{\pm 1}, {1 \over g(x_1)}] \otimes \ldots \otimes \mathbb{F}_p[x_d^{\pm 1}, {1 \over g(x_d)}] .
$$   Set
$$
\widetilde{Q}_0 = Q_0 \times Q_{1,0},
$$
where  $Q_{1,0} = \langle y_1^p, y_2, \ldots, y_d \rangle$.
We extend first the virtual endomorhism $f$ to
$$
\tilde{f} : \widetilde{Q}_0 \to  \widetilde{Q}
$$
by $$\tilde{f}(y_1^p) = y_2, \tilde{f}(y_i) = y_{i+1} \hbox{ for } 2 \leq i \leq d - 1, \tilde{f}(y_d) = y_1.
$$
Define an isomorphism of groups
 $$
 \mu : Q \times \langle S \rangle \to \widetilde{Q}$$
 by $\mu(s_i) = y_i$, $\mu(x_i) = x_i$ for $ 1 \leq i \leq d$.
 We extend the virtual endomorhism $f$ to 
 $$
 \tilde{f} : \widetilde{H} = A_0 S^{-1} \rtimes \widetilde{Q}_0  \to  \widetilde{G}
 $$
 in the following way : 
for $a \in A_0, s \in S$ we choose $ s_1 \in S$ such that $\mu(s s_1) \in \widetilde{Q}_0$ define
$$
\tilde{f}(a/s) = f(a^{s_1})/ \mu^{-1}( \tilde{f}(\mu(s s_1))) \in A S^{-1}
$$
where by abuse of notation we write $a^{ s_1 }$ for $a^{\mu(s_1)}  \in A_0( \mathbb{F}_p Q) = A_0$. By abuse of notation we write $\tilde{f}(s s_1)$ for $\mu^{-1}( \tilde{f}(\mu(s s_1)))$, hence 
$$
\tilde{f}(a/s) = f(a^{s_1})/  \tilde{f}(s s_1) \in A S^{-1}.
$$

\begin{lemma}
$\tilde{f}$ is well-defined. 
\end{lemma}

\begin{proof}
Suppose that $s_1, s_2 \in S$ are such that
$\mu(s s_1) \in \widetilde{Q}_0$ and $\mu(s s_2) \in \widetilde{Q}_0$. Then
$$
f(a^{ s_1})^{ \tilde{f}(s s_2)} = f(a)^{f(s_1) \tilde{f} (s s_2)} =
f(a)^{\tilde{f} (s_1 s s_2)} = f(a)^{f(s_2) \tilde{f} (s s_2)} =
f(a^{s_2})^{\tilde{f} (s s_1)}. 
$$
\end{proof}

\begin{lemma} \label{vir12}  $\tilde{f}$ is a simple, virtual endomorphism.
\end{lemma}

\begin{proof} Note that since $g \not\in (x-1) \mathbb{F}_p[x]$
we have that $A_0 S^{-1} \not= A S^{-1}$, where $A_0 = Aug(\mathbb{F}_p Q)$, otherwise $(A/ A_0) S^{-1} = 0$ and since $A = \mathbb{F}_p + A_0$ there is $s \in S \cap A_0$. By construction $s = s_1^{z_1} \ldots s_d^{z_d} = g(x_1)^{z_1} \ldots g(x_d)^{z_d}$ and then for $x_1 = x_2 = \ldots = x_d = 1$  we get $s(1, \ldots, 1) = g(1)^{z_1 + \ldots + z_d}  \in \mathbb{F}_p \setminus \{ 0 \}$ but on the other hand since $s \in  A_0$ we have $s(1, \ldots, 1)  = 0$, a contradiction.

Let $K$ be a normal subgroup of $\widetilde{G}$ contained in $\widetilde{H} = A_0 S^{-1} \rtimes \widetilde{Q}_0$  such that $\tilde{f}(K) \subseteq K$. Since $K$ is normal in $\widetilde{G}$ we deduce that $K \cap A_0 S^{-1} = K \cap A S^{-1}$ is an ideal in $A S^{-1}$, hence $K \cap A S^{-1} = I S^{-1}$ for some ideal $I$ in $A$ contained in $A_0$. Let
$$
J =  I S^{-1} \cap A = K \cap A_0S^{-1} \cap A  = K \cap A_0 \subseteq A_0.$$
Then
$$
J S^{-1} = I S^{-1} \ \  \hbox{ and }  \ \  J S^{-1} \cap A = I S^{-1} \cap A = J.
$$
Furthermore
$$f(J) =
\tilde{f}(J) \subseteq \tilde{f}(J S^{-1}) = \tilde{f}(I S^{-1}) \subseteq \tilde{f}(K) \subseteq K \ \ 
\hbox{ and } \ \ 
f(J) \subseteq f(A_0) \subseteq A.
$$
Thus
$$f(J) \subseteq A \cap K = A \cap AS^{-1} \cap K = A \cap (K \cap A S^{-1}) = A \cap I S^{-1} = J.
$$
Since $f$ is simple this implies that $J = 0 $ and so $K \cap A_0 S^{-1} = K \cap A S^{-1} = I S^{-1} = J S^{-1} =  0$.
Finally since every non-trivial element of $\widetilde{Q}$ acts non-trivially on  $A S^{-1}$ we deduce that $\tilde{f}$ is simple.
\end{proof}
Lemma \ref{vir12} implies immediately the following corollary.

\begin{cor} $\widetilde{G}$ is a transitive self-similar group.
\end{cor}

\begin{lemma} With the appropriate choice of polynomial $g$ the group $\widetilde{G}$ is finitely presented.
\end{lemma}

\begin{proof}  It suffices to choose $g$ an irreducible polynomial coprime to $x$ such that $g$ is not divisible by $x-1$. It is easy to see that in this case $\widetilde{A}$ is 2-tame as $\mathbb{Z} \widetilde{Q}$-module  and then we can apply Theorem \ref{cond-met}.
\end{proof}

\end{document}